\documentclass[12pt, reqno, a4paper]{article}
\usepackage{amsmath,amssymb,amsthm,todonotes,url,mathtools}
\usepackage[lmargin=35mm,rmargin=35mm,bmargin=30mm,tmargin=30mm]{geometry}
\usepackage{bbm}
\usepackage[inline]{enumitem}

\title{A note on infinite antichain density}
\author{Paul Balister\footnote{Mathematical Institute, University of Oxford,
Oxford OX2\thinspace6GG, United Kingdom  \texttt{\{balister,powierski,scott,jane.tan\}@maths.ox.ac.uk}} \;
Emil Powierski\protect\footnotemark[1] \;
Alex Scott\protect\footnotemark[1]~\footnote{Research supported by EPSRC grant EP/V007327/1.} \;
Jane Tan\protect\footnotemark[1]}
\date{}

\newtheorem{theorem}{Theorem}
\newtheorem{corollary}[theorem]{Corollary}
\newtheorem{claim}{Claim}
\newtheorem*{claim*}{Claim}

\numberwithin{equation}{section}

\newcommand{\floor}[1]{\left\lfloor #1 \right\rfloor}

\newcommand{\F}{\mathcal{F}}

\newcommand{\N}{\mathbb{N}}

\newcommand{\eps}{\varepsilon}

\begin{document}

\maketitle

\begin{abstract}
Let $\F$ be an antichain of finite subsets of $\N$.  How quickly can the quantities $|\F\cap 2^{[n]}|$ grow as $n\to\infty$?
We show that for any sequence $(f_n)_{n\ge n_0}$ of positive integers satisfying $\sum_{n=n_0}^\infty f_n/2^n \le 1/4$, $f_{n_0}=1$ and $f_n\le f_{n+1}\le 2f_n$, there exists an infinite antichain $\F$ of finite subsets of $\N$ such that $|\F\cap 2^{[n]}| \geq f_n$ for all $n\ge n_0$. It follows that for any $\eps>0$ there exists an antichain $\F\subseteq 2^\N$ such that
$$\liminf_{n \to \infty} |\F\cap 2^{[n]}| \cdot \left(\frac{2^n}{n\log^{1+\eps} n}\right)^{-1} > 0.$$ 
This resolves a problem of Sudakov, Tomon and Wagner in a strong form, and is essentially tight.
\end{abstract}

\section{Introduction}
For a set $X$, let $2^X$ denote the power set of $X$, and let $[n]=\{1,\dots,n\}$.
A family $\F$ of sets is an {\em antichain} if $A\not\subseteq B$ for all distinct $A,B \in \F$.  
A well-known theorem of
 Sperner~\cite{sperner} states that any antichain $\F \subseteq 2^{[n]}$  has size
at most $\binom{n}{\floor{n/2}}$; the upper bound is achieved by the antichain consisting of all sets of size~$\floor{n/2}$.
Sperner's theorem is a fundamental result in combinatorics and has led to a huge body of subsequent research (see, for example, \cite{anderson, engel, erdos}).

Now suppose that $\mathcal F$ is an (infinite) collection of finite subsets of the natural numbers.  How fast can $|\F\cap 2^{[n]}|$ grow?  It follows immediately from Sperner's theorem that 
\begin{equation}\label{spbound}
|\F\cap 2^{[n]}|\le \binom{n}{\floor{n/2}}=O(2^n/\sqrt n).
\end{equation}
However, the extremal families for Sperner's theorem for different values of $n$ are far from being nested, so it is not {\em a priori} clear that anything close to this bound can be achieved.

This problem was investigated recently by Sudakov, Tomon, and Wagner~\cite{STW20}.  They show that, in fact, the upper bound on the asymptotic growth rate given by \eqref{spbound} can be improved by a polynomial factor.
Indeed, they note that the following upper bound follows easily from Kraft's inequality~\cite{kraft}.

\begin{theorem}[Sudakov, Tomon, and Wagner~\cite{STW20}]\label{STWupper}
Let $\F\subseteq 2^{\N}$ be an antichain. Then 
\begin{align}\label{STWseries}
    \sum_{n=1}^{\infty} \frac{|\F \cap 2^{[n]}|}{2^n}\le 2.
\end{align}
\end{theorem}

It follows immediately that $|\F\cap 2^{[n]}|$ cannot grow as quickly as $2^n/n\log n$, that is,
\begin{align}\label{STWconcl}
\liminf_{n \to \infty} |\F\cap 2^{[n]}| \cdot \left(\frac{2^n}{n\log n}\right)^{-1} = 0.
\end{align}

Turning to lower bounds, Sudakov, Tomon, and Wagner used an argument based on a carefully chosen family of random walks to construct an antichain with asymptotic growth matching \eqref{STWconcl} up to a polylogarithmic term.

\begin{theorem}[Sudakov, Tomon, and Wagner~\cite{STW20}]\label{STWlower}
There exists an antichain $\F\subseteq 2^{\N}$ with
\[\liminf_{n \to \infty} |\F\cap 2^{[n]}| \cdot \left(\frac{2^n}{n\log^{46} n}\right)^{-1} > 0.
\]
\end{theorem}

They go on to speculate that the bound in Theorem~\ref{STWupper} is essentially optimal, and that the exponent $46$ in Theorem~\ref{STWlower} can be improved to $1+\varepsilon$ for any $\varepsilon>0$. We show that this is indeed the case.  In fact we prove a stronger result, giving essentially optimal bounds on the growth rate of $|\F\cap 2^{[n]}|$. Our main theorem uses a condition that matches the form taken by~\eqref{STWseries} and shows that, under natural additional assumptions, any growth rate for which the stated series is convergent can be attained. 

\begin{theorem}\label{thm:maindensity}
Let $(f_n)_{n\ge n_0}$ be a nondecreasing sequence of positive integers for which $f_{n_0}=1$,
\[\sum_{n=n_0}^\infty\frac{f_n}{2^n}\le\frac{1}{4}
\]
and $\frac{f_n}{2^n}$ is nonincreasing \textup{(}so $f_n\le f_{n+1}\le 2f_n$\textup{)}.
Then there exists an antichain $\F \subseteq 2^\N$ such that
\[
|\F\cap 2^{[n]}| \ge f_n
\]
for all $n\ge n_0$.
\end{theorem}

We remark that if $1<f_{n_0}<2^{n_0}/8n_0$, then one obtains the same result provided
\[
 \sum_{n=n_0}^\infty\frac{f_n}{2^n}\le\frac{1}{4}-2n_0\cdot\frac{f_{n_0}}{2^{n_0}}.
\]
Indeed, it is enough to set $f_n=\lceil f_{n_0}/2^{n_0-n}\rceil$ for $n<n_0$ and apply
Theorem~\ref{thm:maindensity} to the new sequence $(f_n)_{n\ge n'_0}$, where $n'_0=\max\{n:f_n=1\}$.

By taking $f_n$ to be about $2^n/(n \log n^{1+\eps})$ for any $\eps >0$, the following result, answering the question of Sudakov, Tomon, and Wagner, is immediate.

\begin{corollary}\label{thm:specdensity}
There exists an antichain $\F \subseteq 2^\N$ such that
\[
|\F\cap 2^{[n]}| =  \frac{2^n}{n\log^{1+o(1)}n}.
\]
\end{corollary}

\section{Antichain construction}
In this section, we prove Theorem~\ref{thm:maindensity}. 
We use standard notation throughout. We identify infinite binary ($\{0,1\}$-)strings with subsets of $\N$ in the usual way, that is, a string $x_1x_2\cdots$ corresponds to the set $\{i\in \N: x_i=1\}$. Recall that in \emph{lexicographic order}, for distinct binary strings $x_1x_2\cdots$ and $y_1y_2\cdots$ we have $x_1x_2\cdots <_{\operatorname{lex}} y_1y_2\cdots$ if $x_i < y_i$, where $i=\min\{j:x_j\neq y_j\}$, and similarly for finite strings. 

The elements of our antichain will each consist of two concatenated parts where the initial segment encodes the number of 1's in the remainder of the string. By construction, these elements (in particular the initial
segments) naturally occur in reverse lexicographic order and are built in blocks of elements with the same initial segment.

The set of strings that we use as initial segments have the property that no string is an initial segment of any other. Such a set is called a \emph{prefix code}. This condition, while being much weaker than that required for an antichain, gets us ``halfway" there, as it ensures that
elements with prefixes earlier in reverse lexicographic order cannot be subsets of those with later prefixes. To obtain our antichain, we will then append strings to each prefix in such a way that later elements cannot be subsets of earlier ones.

\begin{proof}[Proof of Theorem~\ref{thm:maindensity}]
By assumption, all $f_n$ are positive.
Let $k_0=n_0-1$, and for $k\ge k_0$ define
\[
 \ell_k=\max\big\{n:\tfrac{f_n}{2^n}\ge\tfrac{1}{2^{k+1}}\big\}.
\]
We note that $\ell_k$ is well defined as $f_n/2^n\to0$ and $f_n\ge1$ for $n\ge n_0$,
which also gives $\ell_k\ge k+1$. Also, as $f_n$ is nondecreasing, $\ell_{k+1}>\ell_k$.

Define $a_k=\ell_k-k$ for $k\ge k_0$ and note that $a_k>0$.

\begin{claim}
$\sum_{k=k_0}^\infty \frac{a_k}{2^k}\le 1$.
\end{claim}
\begin{proof}
We note that for any $k\ge k_0$ by definition of $\ell_k$ and by monotonicity of $(f_n/2^n)_{n\ge n_0}$, we have $\frac{f_n}{2^n}\ge 2^{-(k+1)}$ for all $n\in(\ell_{k-1},\ell_k]$. Setting $\ell_{k_0-1}=k_0=n_0-1$ we thus get
\[
 \frac{\ell_k-\ell_{k-1}}{2^{k+1}}
 \le \sum_{n=\ell_{k-1}+1}^{\ell_k}\frac{f_n}{2^n}.
\]
Now as
\[
 \sum_{k\ge k_0}\frac{\ell_k-\ell_{k-1}}{2^{k+1}}
 =\Big(\frac{1}{2}-\frac{1}{4}\Big)
 \sum_{k\ge k_0}\frac{\ell_k-k_0}{2^k},
\]
we have
\[
 \sum_{k\ge k_0}\frac{a_k}{2^k}\le \sum_{k\ge k_0}\frac{\ell_k-k_0}{2^k}
 \le 4\sum_{n\ge n_0}\frac{f_n}{2^n}\le 1.\qedhere
\]
\end{proof}

We greedily construct a prefix code $(c_{k,i})_{k\ge k_0, i\in [a_k]}$ consisting of $a_k$ many strings of length~$k$ with the property that the elements are lexicographically decreasing when ordered so that their indices $(k,i)$ are lexicographically increasing. Such a sequence is given by setting $c_{k,i}$ to be the string of length $k$ with digits $c_{k,i}(1), \dots, c_{k,i}(k)$ defined by
\[
 \sum_{j=1}^k\frac{c_{k,i}(j)}{2^j} = 1- s_{k-1}-\frac{i}{2^k}
\]
where $s_k = \sum_{i=k_0}^k a_i/2^i$.
That is, we take $c_{k,i}$ to be the first $k$ binary
digits of the binary representation of the fraction
$1- s_{k-1}-\frac{i}{2^k}$, which is guaranteed to be positive since $\sum_{k=k_0}^\infty a_k/2^k \le 1$. Equivalently, this sequence may be described by starting with the string of length $k_0$ consisting of all 1's, and then each string of length $k \ge k_0$ is obtained by subtracting $1/2^k$ from the previous string considered as a binary expansion of a fraction.
For example, if $k_0=2$, $a_2=1$, $a_3=3$, and $a_4=5$, then the first six strings would be $c_{2,1}=11$, $c_{3,1}=101$, $c_{3,2}=100$, $c_{3,3}=011$, $c_{4,1}=0101$, $c_{4,2}=0100$. 

It is not hard to see that for two distinct strings in the sequence $(c_{k,i})$, at the first position where they differ the earlier string has a 1 and the later one a 0. It follows that the $c_{k,i}$ indeed form a lexicographically decreasing prefix code. 

Now given a particular string $c_{k,i}$ of length~$k$, let $F_{k,i}$ be the set of all binary strings of length $\ell_k$ satisfying the following conditions:
\begin{enumerate}[label={(\arabic*)}]
\itemsep=0mm
\item The first $k$ digits are precisely $c_{k,i}$.
\item There are precisely $i$ many 1's after the $k$th digit.
\item If $k>k_0$, there is at least one 1 after the $\ell_{k-1}$th digit.
\end{enumerate}

We then define the family
\[
\F:= \bigcup_{\substack{k\ge k_0, \\ i\in [a_k]}} F_{k,i}
\]
and view this as a subset of $2^\N$ by filling out the strings with 0's in the usual way.

\begin{claim}
$\F$ is an antichain.
\end{claim}
\begin{proof}
Take any distinct $x=x_1x_2x_3\ldots,y=y_1y_2y_3\ldots \in \F$, say with $x\in F_{k,i}$ and $y\in F_{k',i'}$. If $k=k'$ and $i=i'$, then $x$ and $y$ have the same number of 1's after the $k$th digit. Since $x$ and $y$ are distinct but agree on the first $k$ digits, this means we find $j$ and $j'$ such that $x_j=0, y_j=1, x_{j'}=1$, and  $y_{j'}=0$. Hence we may assume that $k\le k'$, and if $k=k'$, then $i<i'$.

By construction, we have that $c_{k,i}$ appears earlier than $c_{k',i'}$ in reverse lexicographic order. It follows that $x_j=1$ and $y_j=0$, where $j$ is the first position at which $x$ and $y$ differ, and moreover this must occur at some $j\le k$ as the $c_{k,i}$ form a prefix code. In addition, if $k<k'$, then by condition (3) there is some position $j>\ell_{k'-1}$ for which $y_j=1$. But all 1's in $x$ occur within the first $\ell_k\le \ell_{k'-1}$ places, so $x_j=0$. Otherwise, if $k= k'$ and $i<i'$, then by condition (2) this means that $x$ has fewer 1's after digit $k$ than $y$ does, so there is necessarily some position $j$ for which $x_j=0$ and $y_j=1$. Thus, $x$ is neither a subset nor a superset of~$y$.
\end{proof}

\begin{claim}
For each $k\ge k_0$ and $n\in (\ell_{k-1},\ell_k]$ there are at least $2^{n-k}-1$ strings in~$\F\cap 2^{[n]}$.
\end{claim}
\begin{proof}
We proceed by induction on $k$. For $k=k_0$, condition (3) is void. Thus we have $2^{n-k_0}-1$ choices of binary strings $b$ between positions $k_0+1$ and $n$ that have at least one 1. Denoting concatenation of strings by multiplication, for each $b$ there is precisely one corresponding string in $\F$ agreeing with $b$ in these positions, namely, $c_{k_0,i}b$, where
$i$ is the number of 1's in $b$. Note that, since $a_k = \ell_k-k$ for all $k\ge k_0$ by definition, the number of $1$'s in $b$ does not exceed $a_k$, which ensures that $c_{k_0,i}b$ can be found in $\F$.

Now suppose $k>k_0$. Applying the induction hypothesis for $k-1$ and $n'=\ell_{k-1}$ we see we have at least $2^{\ell_{k-1}-(k-1)}-1$ strings that have no 1 after $\ell_{k-1}$, that is, $|\mathcal F\cap2^{[\ell_{k-1}]}|\geq 2^{\ell_{k-1}-(k-1)}-1$.
Now consider the number of strings that have at least
one 1 after $\ell_{k-1}$. We have $2^{n-k}-2^{\ell_{k-1}-k}$ choices of binary
strings $b$ between positions $k+1$ and $n$ such that $b$ has at least one 1 after $\ell_{k-1}$, and, as above,  for each $b$ there is
precisely one corresponding string $c_{k,i}b$ in $\F$ agreeing
with $b$ in these positions.
Since $\ell_k>\ell_{k-1}$, this makes a total of at least
\[
 2^{n-k}-2^{\ell_{k-1}-k}+2^{\ell_{k-1}-(k-1)}-1\ge 2^{n-k}-1
\] strings in $\F\cap2^{[n]}$.
\end{proof}

Finally, for $n\in(\ell_{k-1},\ell_k]$
we have $f_n/2^n<2^{-k}$ by definition of $\ell_{k-1}$. Hence $2^{n-k} > f_n$ so we have constructed an antichain $\F$ that contains at least $2^{n-k} -1 \geq f_n$ strings in~$\F\cap 2^{[n]}$. This concludes the proof of Theorem~\ref{thm:maindensity}.
\end{proof}

\section*{Acknowledgments}
We would like to thank Stijn Cambie for pointing out an error in the original version of this paper, as well as the anonymous referees for their helpful comments.


\begin{thebibliography}{1}

\bibitem{anderson}
{\sc I.~Anderson}, {\em Combinatorics of finite sets}, Oxford Science Publications, The Clarendon Press, Oxford University Press, New York, 1987.

\bibitem{engel}
{\sc K.~Engel}, {\em Sperner theory}, Encyclopedia Math. Appl. 65, Cambridge University Press, Cambridge, 1997.

\bibitem{erdos}
{\sc P.~Erd\H{o}s}, {\em On a lemma of {L}ittlewood and {O}fford}, Bull. Amer. Math. Soc., 51 (1945), pp.~898--902.

\bibitem{kraft}
{\sc L.~G. Kraft}, {\em {A device for quantizing, grouping, and coding amplitude-modulated pulses}}, Thesis (M.S.) Massachusetts Institute of
  Technology. Dept. of Electrical Engineering, 1949.

\bibitem{sperner}
{\sc E.~Sperner}, {\em {Ein Satz \"{u}ber Untermengen einer endlichen Menge}}, Math. Z., {27} (1928), pp.~544--548.

\bibitem{STW20}
{\sc B.~Sudakov, I.~Tomon, and A.~Z. Wagner}, {\em {Infinite Sperner's theorem}}, J. Combin. Theory Ser. A, {187} (2022), 105558.
\end{thebibliography}
\end{document}